\documentclass[smallextended]{svjour2}
\smartqed
\usepackage{mathptmx, amsmath, amsfonts, color,  epsfig,graphicx}
\journalname{Numerische Mathematik}

\begin{document}
\title{A hybridizable discontinuous Galerkin  method  for fractional diffusion problems}
\titlerunning{HDG for Fractional Diffusion Problems}
\author{Bernardo Cockburn \and Kassem Mustapha\thanks{The valuable comments of the editor and the referees improved the paper. The support of  the Science Technology Unit at KFUPM through  King Abdulaziz City for Science and
Technology (KACST) under National Science, Technology and Innovation Plan (NSTIP) project No. 13-MAT1847-04
 is gratefully acknowledged.}}
\institute{Bernardo Cockburn \at School of Mathematics, University of Minnesota, USA ({\tt cockburn@math.umn.edu}). \and Kassem Mustapha \at
Department of Mathematics and Statistics, King Fahd University of Petroleum and Minerals (KFUPM), Saudi Arabia ({\tt kassem@kfupm.edu.sa}).}

\date{\today}
\maketitle
\begin{abstract}
We study the use of the hybridizable discontinuous Galerkin (HDG) method for numerically solving fractional diffusion equations of order
$-\alpha$ with  $-1<\alpha<0$. For exact time-marching, we derive optimal algebraic error estimates  {assuming} that the exact solution is
sufficiently regular. Thus, if for each time $t \in [0,T]$ the approximations are taken to be piecewise polynomials of degree $k\ge0$ on the
 spatial domain~$\Omega$, the approximations to  $u$ in the $L_\infty\bigr(0,T;L_2(\Omega)\bigr)$-norm   and to $\nabla u$
 in the $L_\infty\bigr(0,T;{\bf L}_2(\Omega)\bigr)$-norm are
 proven to converge with
the rate $h^{k+1}$, where $h$ is the maximum diameter of the elements of the mesh.  Moreover, for $k\ge1$ and quasi-uniform meshes, we obtain a
superconvergence result which allows us to compute, in an elementwise manner, a new approximation for $u$ converging with a rate of
$\sqrt{\log(T h^{-2/(\alpha+1)})}\, \,h^{k+2}$.
\end{abstract}

\newcommand{\Red}[1]{{\color{red}{#1}}}
%
%

%
%
\newcommand{\B}{\mathcal{B}}
\newcommand{\I}{\mathcal{I}}
%
\newcommand{\marg}[1]{\marginpar{\tiny{\framebox{\parbox{1.7cm}{#1}}}}}

\newcommand{\Hdiv}[1]{{\bf{H}}(\mathrm{div},{#1})}
\def\d{\partial}
\newcommand{\RRR}{\mathbb{R}}

\newcommand{\CP}{C_{\!\scriptscriptstyle{\pr}}}
\newcommand{\CR}{C_{\!\mathrm{reg}}}

\newcommand{\trF}{\jb{\mathop{\gamma_F}}}
\newcommand{\tr}[1]{\mathop{\mathrm{tr}_{#1}}}
\newcommand{\trnorm}[1]{{|\!|\!|\,{#1}\,|\!|\!|}}
\newcommand{\E}{\mathcal{E}}
\newcommand{\veps}{\varepsilon}
\newcommand{\eq}{\boldsymbol{\varepsilon}_h^{\;q}}
\newcommand{\eqhat}{\boldsymbol{\varepsilon}_h^{\widehat{q}}}
\newcommand{\eu}{\varepsilon_h^u}
\newcommand{\esig}{\B_{\alpha}\eq}
\newcommand{\el}{\varepsilon_h^{\widehat u}}
\newcommand{\vPhi}{\boldsymbol{\varPhi}}
\newcommand{\vTheta}{\varTheta}
\newcommand{\vPsi}{\varPsi}
\newcommand{\Poly}{\mathbb{P}}
\newcommand{\taum}{\tau_K^{\mathrm{max}}}
\newcommand{\tm}{\tau^{\mathrm{max}}}
\newcommand{\diam}{\ensuremath{\mathop\mathrm{diam}}}
\newcommand{\meas}{\ensuremath{\mathop\mathrm{meas}}}
\newcommand{\Eh}{\mathcal{E}_h}
\newcommand{\oh}{{\mathcal{T}_h}}
\newcommand{\doh}{{\partial \oh}}
\newcommand{\Th}{{\oh}}
\newcommand{\qhat}{\widehat{{\bf q}}}
\newcommand{\sighat}{\B_\alpha\widehat{{\bf q}}}
\newcommand{\uhat}{\widehat{u}}
\newcommand{\Uhat}{\widehat{U}}
\newcommand{\eh}{\boldsymbol{\veps}^{\qhat}_h}
\newcommand{\esigh}{\B_{\alpha}\boldsymbol{\veps}^{\widehat {{\bf q}}_h}}
\newcommand{\ip}[1]{\langle {#1} \rangle}
\newcommand{\Forall}{\quad\text{for all }}
\newcommand{\pr}{\varPi}
\newcommand{\pv}{\boldsymbol{\varPi}_{\!\scriptscriptstyle{V}}}
\newcommand{\pw}{\varPi_{\!\scriptscriptstyle{W}}}
\newcommand{\Pw}{P_{\scriptscriptstyle{W}}}
\newcommand{\Pm}{P_{\scriptscriptstyle{M}}}
\newcommand{\vdelta}{{\boldsymbol{\delta}}}
\newcommand{\ve}{{\boldsymbol{e}}}
\newcommand{\vp}{{{\bf P}}}
\newcommand{\vP}{{{\bf P}}}
\newcommand{\vq}{{{\bf q}}}
\newcommand{\vQ}{{{\bf q}}}
\newcommand{\vsig}{\B_{\alpha}{{\bf q}}}
\newcommand{\vSig}{{\boldsymbol{\Sigma}}}
\newcommand{\vr}{{\bf {r}}}
\newcommand{\vz}{{\bf {z}}}
\newcommand{\vn}{{\bf {n}}}
\newcommand{\vv}{{\bf {v}}}
\newcommand{\vw}{{\bf {w}}}
\newcommand{\dive}{\mathop{\nabla}\cdot\,}
\newcommand{\pol}{\mathcal{P}}
\newcommand{\bpol}{\boldsymbol{\mathcal{P}}}
\section{Introduction}
In this paper, we propose and analyze a numerical method using exact integration in time and the so-called HDG method for the spatial
discretization of the following anomalous, slow diffusion (sub-diffusion) model problem:
\begin{subequations}
\label{heatequation}
\begin{alignat}{2}\label{heatequation-a}
&u_t-\B_\alpha \Delta u = f &&\quad\mbox{ in }\Omega\times (0,T],
\\
&u= g &&\quad\mbox{ on }\partial\Omega\times (0,T],
\\  \label{heatequation-b}
&u|_{t=0}=u_0 &&\quad\mbox{ on }\Omega,
\end{alignat}
\end{subequations}
{where $\Omega$ is a convex polyhedral domain of $\mathbb{R}^d$, where $d=1,2,3$.} Here, $\B_\alpha$ is the Riemann--Liouville fractional
derivative in time defined, for $-1<\alpha<0$, by
\begin{equation}
\label{Balpha} \B_\alpha v(t):=\frac{\partial }{\partial t}\int_0^t\omega_{\alpha+1}(t-s)v(s)\,ds\quad\text{with} \quad
\omega_{\alpha+1}(t):=\frac{t^\alpha}{\Gamma(\alpha+1)}
\end{equation}
where  $\Gamma$ denotes the usual gamma function. One may show that $\B_\alpha v\to v$ as $\alpha \to 0$. So,  in the limiting case~$\alpha=0$,
the problem \eqref{heatequation} becomes nothing but an initial-boundary value probem for the classical heat equation.

Problems of the form \eqref{heatequation} arise in a variety of physical, biological and chemical applications
\cite{kilbas,Mathai,MetzlerKlafter2000,MetzlerKlafter2004,podlubny,SmithMorrison,Tarasov}. They describe slow or anomalous sub-diffusion and
occur, for example, in models of fractured or porous media, where the particle flux depends on the entire history of the density gradient,
$\nabla u$. It is thus important to devise, efficient methods for numerically solving them.

Let us briefly review the development of numerical methods for the fractional sub-diffusion problem~\eqref{heatequation}. Several authors have
proposed a variety of numerical methods for this problem. For finite difference (FD) methods with convergence rates of order $O(h^2)$ in space,
where $h$ is the maximum meshsize, see, for example,
\cite{ChenLiuAnhTurner2011,ChenLiuTurnerAnh2010,LanglandsHenry2005,LiuYangBurrage2009,MustaphaAlMutawa,Yuste2006,YusteAcedo2005,ZhuangLiuAnhTurner2008,ZhuangLiuAnhTurner2009}.
In \cite{Cui2009}, FD schemes were considered which are first-order accurate in time but $O(h^4)$-accurate in space provided $u$ is sufficiently
smooth including at $t=0.$ In \cite{Mustapha2011}, the second author studied a FD method in time combined with spatial piecewise linear finite
elements scheme.
 In~\cite{McLeanMustapha2009,MustaphaMcLean2011,MustaphaMcLean2013},
a piecewise-constant and a piecewise-linear, discontinuous Galerkin (DG)
 and a postprocessed DG  time-stepping methods combined with piecewise-linear
 finite elements for the spatial discretization were  analyzed. Full convergence  results were
provided for variable time steps employed to compensate the lack of regularity of the exact solution near $t=0$. A FD method and convolution
quadrature had been studied in \cite{CuestaLubichPalencia2006,SchaedleLopezFernandezLubich2006}. Another type of scheme involving Laplace
transformation combined with a quadrature along a contour in the complex plane, provides spectral accuracy for the time discretization, but
appears to offer little scope for handling nonlinear versions of~\eqref{heatequation}, see
\cite{LopezFernandezPalenciaSchaedle2006,McLeanThomee2010L}.

Furthermore, various numerical methods have been applied for the following  alternative representation of  the fractional sub-diffusion equation
(\ref{heatequation-a}): {
\[
\int_0^t\omega_{-\alpha}(t-s)u_t(s)\,ds - \Delta u(t) = \mathsf{f}(t)
\quad\mbox{ in }\Omega\times (0,T],
\]
see \cite{Cui2012,GaoSun2011,JinLazarovZhou2012,Quintana-MurilloYuste2011,ZhangSun2011} and the references therein. The two representations are
equivalent under reasonable assumptions on the initial data, see \cite{YusteQuintana2009}, but the  methods obtained for each representation are
formally different.

Here, we continue the above-described effort and propose and analyze a method using exact integration in time and the HDG method for the space
discretization for problem \eqref{heatequation}.
 The choice of the HDG methods for the problem under consideration
can be easily justified. Indeed, the HDG methods are a relatively new class of { DG} methods introduced in
\cite{CockburnGopalakrishnanLazarov09} in the framework of steady-sate diffusion which share with the classical (hybridized version of the)
mixed  finite element methods their remarkable convergence properties,
\cite{CockburnGopalakrishnanSayas09,CockburnQiuShi,CockburnQiuShiCurved}, as well as the way in which they can be efficiently implemented,
\cite{KirbySherwinCockburn}. They provide approximations that are more accurate than the ones given by any other DG method for second-order
elliptic problems \cite{NguyenPeraireCockburnIcosahom09}.

Here we  prove that, for each time $t \in [0,T]$, the error of the HDG approximation to the solution $u$ of \eqref{heatequation} in the
$L_\infty\bigr(0,T;{L}_2(\Omega)\bigr)$-norm and to the flux ${\bf q}:=-\nabla u$  in the $L_\infty\bigr(0,T;{\bf L}_2(\Omega)\bigr)$-norm
converge with order  $h^{k+1}$ where $k$ is the polynomial degree; see Theorem \ref{thm:ee}. We also show that a suitably defined {\em
projection} of the error in $u$ {\em
  superconverges}  with order
$h^{k+2}$ whenever $k\ge1$. This allows us to obtain, by a simple elementwise postprocessing, another approximation to $u$ converging in the
$L_\infty \bigr(0,T;L_2(\Omega)\bigr)$-norm with a rate of  $\sqrt{\log (T/h^{2/(\alpha+1)}) }h^{k+2}$ for quasi-uniform meshes and whenever
$k\ge1$; see Theorem \ref{thm:super}. We thus obtain a much better approximation at a cost which is negligible in comparison with that of
obtaining the approximate solution.  These convergence results extend those obtained in \cite{ChabaudCockburn12} for the heat equation, that is
for the case $\alpha=0$, and hold uniformly for any  $-1< \alpha\le 0$.  Our error analysis extends the approach used in
\cite{ChabaudCockburn12} for the heat equation. We make the full use of several important properties of the fractional derivative operator
$\B_\alpha$; see Lemma \ref{appendix}. In particular, especial care has to be used in the proof of the uniformity-in-time of the above-mentioned
superconvergence property, as new, delicate regularity estimates are required by the use of a fractional duality argument.

 Outline of the paper. In the next section, we
define the  HDG method.  In Section~\ref{sec: errors}, we prove the main
convergence result,  Theorem \ref{thm:ee}. Particularly relevant to this a priori error analysis is the derivation of several important
properties of the fractional order operator $\B_\alpha$, which we gather in Lemma \ref{appendix}. In Section~\ref{sec:super}, we prove the
superconvergence result, Theorem \ref{thm:super}.
Finally, in Section \ref{sec:extensions}, we comment on the extension of this work to other methods
 fitting the general formulation of the HDG methods; see \cite{CockburnQiuShi}.

\section{The HDG method}\label{sec: HDG scheme}
 We begin this section  by discretizing the domain $\Omega$ by a
 triangulation $\oh$ (made of simplexes $K$) which we take to be conforming for the sake of simplicity. We denote by $\partial
 \oh$ the set of all the boundaries $\partial K$ of the elements $K$ of $\oh$.  We denote by $\Eh$ the union of faces $F$ of the simplexes $K$ of the triangulation
$\oh$.

Next, we introduce the discontinuous finite element spaces:
\begin{subequations}
\begin{alignat}{3}
  \label{eq:W}
  {W_h} &= \{ {w \in L^2(\Omega)}&&{: w|_K \in \pol_k(K)}&&\;\;\;{\forall\;K\in\oh}\},
  \\
  \label{eq:V}
  {\bf V}_h &=  \{ \vv\in {\bf L}^2(\Omega):=[L_2(\Omega)]^d&&{:\; \vv|_K \in \bpol_k(K)}&&\;\;\;{\forall\; K\in\oh}\},
  \\
  \label{eq:M}
  M_h &= \{ \mu\in L^2(\Eh)&&: \;\mu|_F \in
  \pol_k(F)&&\;\;\;\forall\; F\in \Eh\},
\end{alignat}
\end{subequations}
where  $\bpol_k(K):=[\pol_k(K)]^d$ (the space of
  vector-valued functions whose entries lie on $\pol_k(K)$).
Here,  $\pol_k(D)$ is the space of polynomials of total degree $\le k$ on any spatial domain $D$.

To describe our HDG scheme, we rewrite (\ref{heatequation-a}) as a first order system as follows}: $\vq +\nabla u = 0$, $u_t+\dive \vsig = f$ in
$\Omega\times (0,T].$ So, the exact solution  satisfies:
\begin{subequations}
\label{eq: weak exact solution}
  \begin{alignat}{2}\label{eq: weak exact solution1}
    ( \vq,{\bf \phi}) -
    (  u,\dive {\bf \phi}) +   \ip{ u,{\bf \phi}\cdot\vn}
    &  = 0 && \quad \forall {\bf \phi}\in \Pi_{K\in\mathcal{T}_h}{ \Hdiv{K}},
    \\\label{eq: weak exact solution3}
     (u_t,\chi) -(   \vsig, \nabla \chi) +
    \ip{  \vsig\cdot\vn, \chi}
    & =
    (f, \chi)&&\quad \forall \chi \in \Pi_{K\in\mathcal{T}_h} H^1(K)\,.\end{alignat}
\end{subequations}
for $t \in (0,T]$, where  $(v,w):=\sum_{K \in \Th}(v,w)_K$ and $\ip{v,w} := \sum_{K \in \Th} \ip{ v,w}_{\d K}$. We  write, for any domain $D$ in
$\RRR^d$, $(u,v)_D :=\int_D u v \; dx$, and $\ip{u,v}_{\d D}:=\int_{\d D} u ,v\;d\gamma$. For vector-valued functions $\vv$ and $\vw$, the
notation is similarly defined with the integrand being the dot product $\vv \cdot \vw$.

The HDG method {provides} a scalar approximation $u_h(t) \in W_h$ to $u(t)$,
 a vector-valued approximation  $\vq_h(t) \in  {\bf V}_h$  to the flux  $\vq(t)$,
and {a scalar approximation $\uhat_h(t)\in M_h$ to the {trace} of $u(t)$ {on} element {boundaries}} for each time $t \in [0,T],$  which are
determined by requiring that the equations
\begin{subequations}
\label{method}
\begin{alignat}{2}
\label{eq:method-a}
 ( \vq_h,\vr) -
 ( u_h,\dive \vr) +   \ip{\uhat_h,\vr\cdot\vn}
 &  = 0,
 \\
 \label{eq:method-b}
(\partial_t u_h, w) -( \vsig_h, \nabla w) +
 \ip{ \B_{\alpha}\widehat{{\bf q}}_h\cdot\vn, w}
 & =
 (f, w),
\\
\label{eq:method-d}
 \ip{\uhat_h,\mu }_{\d\Omega}
& = \ip{g, \mu }_{\d\Omega},
\\
\label{eq:method-c}
 \ip{\B_{\alpha}\widehat{{\bf q}}_h \cdot \vn,\mu }
   -\ip{\B_{\alpha}\widehat{{\bf q}}_h \cdot \vn,\mu }_{\d\Omega} & = 0,
\\
\label{eq:method-e}
 u_h|_{t=0} & = \pw u_0,
\end{alignat}
hold  for all $\vr\, {\in} {\bf V}_h,$ $ w \in W_h,$ and $ \mu \in M_h.$  Here, $\partial_t u_h$ is nothing but the partial derivative of
  $u_h$ with respect to time. We
use the notation  $(v,w)_{\oh}:=\sum_{K \in \Th}(v,w)_K$
and $\ip{v,w}_{\doh} := \sum_{K \in \Th} \ip{ v,w}_{\d K}$, and take the numerical trace for the flux as
\begin{alignat}{2}
\label{HDGtrace}  \qhat_h &= \vq_h + \tau \, \big( u_h - \uhat_h \big) \vn  &&\quad \text{ on } \doh,
\end{alignat}
\end{subequations}
for some nonnegative stabilization function $\tau$ defined on $\doh$; we assume that, for each element $K\in\oh$, $\tau|_{\partial K}$ is
constant on each of its faces. How to choose this stabilization
  function in order to achieve optimal convergence properties is dicussed later. Note that the first two equations are inspired in the weak form  satisfied by the exact solution, \eqref{eq: weak
exact solution}. The operator $\pw$ is the one introduced in \cite{CockburnGopalakrishnanSayas09} and will be defined later.

Let us briefly describe the feature of the HDG method which renders it efficiently implementable. Note that the form of the numerical trace
given by \eqref{eq:method-c} allows us to express $(u_h,\vq_h,\widehat{{\bf q}}_h)$ elementwise in terms of $\uhat_h$, $f$ and  $u_0$ by using
equations \eqref{eq:method-a}, \eqref{eq:method-b}, \eqref{HDGtrace} and \eqref{eq:method-e}. Then, $\uhat_h$ is determined by as the solution
of the transmission condition \eqref{eq:method-c}, which enforces the single-valuedness of the normal component of the numerical trace
$\B_{\alpha}\widehat{{\bf q}}_h$, and the boundary condition \eqref{eq:method-d}. Thus, the only globally-coupled degrees of freedom are those
of the numerical trace $\uhat_h$.

Let us end this subsection by noting that the existence and
  uniqueness of the approximation provided by the HDG method just introduced
follows from the corresponding results for linear systems of fractional
differential equations. In particular, see \cite{kilbas} in page 139 the result
for the Cauchy-problem
for the linear system (3.1.29).

\section{Properties of the operator $\B_\alpha$}\label{sec:Balpha}
We begin the analysis by collecting several crucial properties of the operator $\B_\alpha$. They involve the adjoint operators $\B_\alpha^*$ and
$\I^*_{-\alpha}$ of $\B_\alpha$ and $\I_{-\alpha}$, respectively,
 where  $\I_{-\alpha}$ is the Riemann--Liouville fractional integral;
 \[ \I_{-\alpha} v(t)=\int_0^t\omega_{-\alpha}(t-s) v(s)\,ds
    \quad\text{for $-1<\alpha<0$.}
\]
As we pointed out in the Introduction, these properties are essential for the analysis because they allow us to extend the approach used for the
error analysis of the HDG method applied to the heat equation considered in \cite{ChabaudCockburn12}.

For convenience, we introduce the following notation. Starting from the definition of the adjoint operators $\B_\alpha^*$ and $\I_{-\alpha}^*$,
\begin{subequations}
\label{eq: adjoint}
\begin{alignat}{1}
\label{eq: adjointB} \int_0^T v(t)\,\B_\alpha w(t)\,dt
    &=\int_0^T \B_\alpha^*v(t)\, w(t)\,dt,
\\
\label{eq: adjointI} \int_0^T v(t)\,\I_{-\alpha} w(t)\,dt
    &=\int_0^T \I_{-\alpha}^*v(t)\, w(t)\,dt,
\end{alignat}\end{subequations}
one can show that for { $\alpha\in (-1,0)$ and $t\in (0,T]$, see  \cite[Lemma 3.1]{MustaphaMcLean2013}, that
\begin{subequations}
\begin{alignat}{2}
\label{Balpha*} \B_\alpha^*v(t)&=-\frac{\partial}{\partial t}\int_t^T
    \omega_{1+\alpha}(s-t)\,v(s)\,ds
&&\quad\text{for~any $v\in  \mathcal C^1(0,T)$},
\\
\label{Ialpha*} \I^*_{-\alpha} v (t)&=\;\int_t^T\omega_{-\alpha}(s-t)\,v(s)\,ds &&\quad\text{for~any $v\in \mathcal C^0(0,T)$}\,.
\end{alignat}
\end{subequations}
Moreover, since
 \begin{alignat*}{1}
 \B^*_\alpha\I^*_{-\alpha}v(t)
 &=-\frac{\partial}{\partial t}\int_t^T
    \omega_{1+\alpha}(s-t) \int_s^T\omega_{-\alpha}(q-s)\,v(q)\,dq\,ds\\
&=-\frac{\partial}{\partial t}\int_t^Tv(q)\int_t^q
    \omega_{1+\alpha}(s-t)\,\omega_{-\alpha}(q-s)\,ds\,dq,
    \end{alignat*}
and since  $\int_t^q \omega_{1+\alpha}(s-t)\,\omega_{-\alpha}(q-s)\,ds=1$,
it is easy to see that $\I_{-\alpha}^*$ is the {\it right-inverse} of $\B^*_\alpha$, that is,
\begin{equation}\label{eq: right-inverse}
\B^*_\alpha\I^*_{-\alpha}v = v.
\end{equation}}

 We gather in the following result several key
properties we use in our analysis. They are expressed by using a notation we introduce next. First, we set
\[
| v|^2_{\beta,\tilde t }:= \int_0^{\tilde t} v\,\B_{\beta}v\,dt \quad\mbox{ if }\beta\in (-1,0] ~~{\rm and}~~
 | v|^2_{\beta,\tilde t }:= \int_0^{\tilde
t}v\,\I_{\beta}v\,dt \quad\mbox{ if }\beta\in [0,1).
\]
and use the standard notation of the seminorm $|\cdot|$ because, as we are going to see, the two right-hand sides are actually nonnegative. The
term $\int_0^{\tilde t} v(t)\,\I_{\beta}v(t)\,dt$ is nonnegative if $v\in L_2(0,\tilde t)$. The term $\int_0^{\tilde t}
v(t)\,\B_{\beta}v(t)\,dt$ is nonnegative when $v$ is in $\mathcal C^1(0,\tilde{t})$, or, alternatively, when $v$ and $\B_\alpha v$ are $\mathcal
C^0(0,\tilde t);$ see \cite[Equation 6]{McLeanMustapha2009}.

 Finally, for a  given function $v$ defined on $[0,\tilde t]\times \oh$, we set
\[
\| v\|^2_{\beta,\tilde t}:=
\begin{cases}
\int_0^{\tilde t}(\B_{\beta}v,v)\,dt &\quad\mbox{ if }\beta\in (-1,0],
\\
\int_0^{\tilde t}(\I_{\beta}v,v)\,dt &\quad\mbox{ if }\beta\in [0,1).
\end{cases}\]
For  functions defined on $[0,\tilde t]\times \doh$, we replace $\| \cdot \|$ and $(\cdot,\cdot)$ with $\trnorm{\cdot}$ and $\ip{\cdot,\cdot}$,
respectively. Note that, we drop out $\tilde t$   from the above definitions when $\tilde t = T$.

{ \begin{lemma} \label{appendix} Let $c_\alpha =\frac{ \cos( \alpha\pi/2)}{\pi^\alpha}\frac{|\alpha|^{-\alpha}}{(1-\alpha)^{1-\alpha}}$ and $
d_\alpha={1}/{\cos(\alpha \pi/2)}\,.$ Then,
 for any $v,\,w  \in \mathcal C^1(0,T)$ and any $\alpha\in (-1,0)$, we have
\begin{itemize}
\item[{\rm (i)}]
$|v|_\alpha^2 \ge c_\alpha T^\alpha \int_0^T v^2(t)\,dt,$
\item[{\rm (ii)}]
$ \int_0^T v(t)\,w(t)\,dt
    \le { d_\alpha} |v|_\alpha\,|w|_{-\alpha}, $
\item[{\rm (iii)}]
$ \int_0^Tv(t)\,\B_\alpha w(t)\,dt
    \le {d_\alpha}|v|_\alpha\,|w|_\alpha,
$
\item[{\rm (iv)}]
$ \int_0^T \I_{-\alpha}v(t)\,w(t)\,dt
    \le { d_\alpha} |v|_{-\alpha}\,|w|_{-\alpha},\quad{\rm for~any}~~ v,\,w  \in \mathcal C^0(0,T)
$
\item[{\rm (v)}]
$ \lim_{t\downarrow 0} \omega^{-1}_{\alpha+2}(t) \int_{0}^tv(s)\B_\alpha v (s)\,ds =v^2(0).$
\end{itemize}
\end{lemma}}
\begin{proof}
The coercivity property (i) was proven in \cite[Theorem A.1]{McLean2012} by using the Laplace transform and Plancherel Theorem. Using a similar
technique { and  the fact that $\I_{-\alpha}^*$ is the right-inverse of $\B^*_\alpha$, see \eqref{eq: right-inverse}}, property (ii) can also be
obtained, see \cite[Lemma 3.1]{MustaphaSchoetzau2013}. Properties (iii) and (iv) easily follow from property (ii) and { again} from the fact
that $\I_{-\alpha}^*$ is the right-inverse of $\B^*_\alpha$.

It remains to prove property (v). We have, for small enough $t>0$, that
\begin{alignat*}{1}
&\omega^{-1}_{\alpha+2}(t) \int_{0}^tv(s)\B_\alpha v (s)\,ds =\omega^{-1}_{\alpha+2}(t) \int_{0}^t\omega_{\alpha+1}(s)\,v(s)
                                  \omega^{-1}_{\alpha+1}(s)\B_\alpha v (s)\,ds\\
&\quad=\big[\omega^{-1}_{\alpha+2}(t) \int_{0}^t\omega_{\alpha+1}(s)\,ds\big] \,v(t^*)\,\omega^{-1}_{\alpha+1}(t^*)\B_\alpha v (t^*)
=v(t^*)\,\omega^{-1}_{\alpha+1}(t^*)\B_\alpha v (t^*),
\end{alignat*}
for some $t^*\in (0,t)$. { From the definition of $\B_\alpha$, \eqref{Balpha}, we have that
\begin{alignat*}{1}
\B_\alpha v (t^*) &=\omega_{\alpha+1}(t^*)\, v(0)+\int_0^{t^*}\omega_{\alpha+1}(s)v'(t^*-s)\,ds\,.
\end{alignat*}
Since $\int_0^{t^*}\omega_{\alpha+1}(s)|v'(t^*-s)|\,ds< \infty$, the desired result follows.} $\quad \qed$
\end{proof}

\section{Error estimates}
\label{sec: errors} In this section, we carry out the first part of our a priori error analysis of the HDG method. To be able to do this, we
carefully use several crucial properties of the operators $\B_\alpha$ and  $\I_{-\alpha}$ introduced in the previous section.\\

{\bf  4.1 Projections}
 Given $\vq\in {\bf H}^1(\oh):=\prod_{K\in \oh} {\bf H}^1(K)$ and $u\in  H^1(\oh):=\prod_{K\in \oh} H^1(K)$,
the projections $ \pv\vq \in {\bf V}_h$ and $\pw u \in W_h$ are on each simplex $K\in\oh$ as the solutions of the following equations:
\begin{subequations}
  \label{eq:proj}
  \begin{align}
  \label{eq:proj1 new}
    (\pv \vq, \vv)_K
    &= ( \vq, \vv)_K
    && \Forall \vv \in \bpol_{k-1}(K),
    \\
    \label{eq:proj2}
    (\pw u, w )_K
    &= ( u, w )_K
    && \Forall w \in \pol_{k-1}(K),
    \\
    \label{eq:proj3}
    \ip{\pv \vq \cdot \vn + \tau \pw u,\mu }_F
    &= \ip{\vq\cdot\vn + \tau  u,\mu }_F
    && \Forall \mu \in \pol_k(F),
  \end{align}
\end{subequations}
for all faces $F$ of the simplex $K$. This is the projection introduced in \cite{CockburnGopalakrishnanSayas09} to study HDG methods for the
steady-state diffusion problem. Its approximation properties are described in the following result. For convenience, we introduce the following
notation: $e_\vq:= \pv \vq-\vq$ and $e_u:= \pw  u - u.$

 We use  $\| \cdot\|_D$ to denote the $L^2(D)$-norm.  The norm on any other Sobolev space $X$ is denoted by $\|\cdot\|_X$. We also denote $\|
\cdot\|_{X(0,T;Y(D))}$ by $\| \cdot\|_{X(Y(D))}$ and omit $D$ whenever $D=\Omega$.

\begin{theorem}[\cite{CockburnGopalakrishnanSayas09}]
\label{thm:proj} Suppose
   {$\tau|_{\d K}$ is nonnegative and $\tau^{\max}_K:=\max\tau|_{\d K} >0$}.
  Then the system~\eqref{eq:proj} is uniquely solvable for
$\pv \vq$ and~$\pw u$.  Furthermore, there is a constant ${C}$
  independent of $K$ and $\tau$ such that
{
  \begin{alignat*}{1}
   \| e_\vq \, \|_K\le&\;C\,h^{k+1}_K\left(|\vq|_{{\bf H}^{k+1}(K)}
                       +{\tau_K^{*}}\,| u|_{H^{k+1}(K)}\right),
    \\
    \| e_u \|_K \le&\;C\,h^{k+1}_K\left(|u|_{H^{k+1}(K)}
                       +|\nabla\cdot\vq|_{H^k(K)}/\tau_K^{\max}\right)\,.
  \end{alignat*}
 Here $\tau_K^{*}:=\max \tau|_{\d K\setminus F^*}$, where $F^*$ is a face of $K$ at which $\tau|_{\d K}$ is maximum.}
\end{theorem}

Note that  the approximation error of the projection is of order $k+1$  provided that  the stabilization function is such that both  $\tau^*_K$
and $1/\tau^{\max}_K$ are uniformly bounded and the exact solution is
sufficiently regular. For example, we can take $\tau$ to be a
 positive constant. Another possible choice is to take it zero on all but
one face of the simplex $K$, so that $\tau^*_K=0$, and then take it equal to
$1/h_K$ on the remaining face, so that $1/\tau^{\max}_K=h_K$.\\

{\bf 4.2 The equations of the projection of the errors}
 Setting  \begin{equation} \label{eq: comparison}
  (\eq,\eu, \el, \eqhat):=  (\pv \vq-\vq_h, \pw u-u_h, P_M u-\widehat{u}_h, {\bf P}_M \vq-\widehat{\vq}_h),
\end{equation}
where $P_M$ denotes the $L^2$-orthogonal projection onto $M_h$, and ${\bf P}_M$ denotes the vector-valued projection each of whose components
are equal to $P_M$. The projection of the errors satisfy the following equations:
\begin{lemma}        \label{lem:consistency}
We have
  \begin{subequations}
    \label{eq:err}
    \begin{alignat}{2}
      \label{eq:err-a}
 ( \eq,\vr) -
 ( \eu,\dive \vr)_{\oh} +   \ip{ \el,\vr\cdot\vn}
 &  = ( e_{\vq} ,\vr),
 \\
      \label{eq:err-b}
 (\partial_t \eu,w)-( \esig, \nabla w)_{\oh} +   \ip{ \B_\alpha \eqhat \cdot\vn, w}
 & = (e_{u_t},w),
\\
    \label{eq:err-e}
 \ip{\el,\mu}_{\d\Omega}
   & = 0,
\\
    \label{eq:err-c}
 \ip{\B_\alpha  \eqhat \cdot \vn,\mu }-\ip{\B_\alpha  \eqhat \cdot \vn,\mu }_{\d\Omega}
   & =  0,
\\
    \label{eq:err-f}
 \eu|_{t=0} & = 0,
\end{alignat}
for all  $\vr\in {\bf V}_h, \, w\in W_h,$ and $\mu\in M_h$,
  where
\begin{alignat}{2}
  \label{eq:err-d2}
  \eqhat\cdot\vn
                 &:= \eq\cdot\vn  + \tau (\eu -\el)
&&\quad\mbox{ on }\doh.
\end{alignat}
\end{subequations}
\end{lemma}
\begin{proof} From \eqref{eq: weak exact solution}, we know that the exact solution $\{\vq,u\}$ satisfies the equations
  \begin{alignat*}{2}
    ( \vq,\vr) -
    (  u,\dive \vr) +   \ip{ u,\vr\cdot\vn}
    &  = 0 &&\quad\mbox{ for all $\vr\in {\bf V}_h$},
    \\
     (u_t,w)-(   \vsig, \nabla w) +
    \ip{  \vsig\cdot\vn, w}
    & =
    (f, w) &&\quad\mbox{ for all $w \in W_h$}\,.
\end{alignat*}
By using the orthogonality  properties of the projections $\pv$, $\pw$, and $\Pm$, we can rewrite these equations as follows:
  \begin{alignat*}{2}
    ( \pv\vq,\vr) -
    ( \pw u,\dive \vr) +   \ip{\Pm u,\vr\cdot\vn} &  = (e_{\vq} ,\vr),
    \\
    (\pw u_t,w)-(  \B_\alpha\pv\vq, \nabla w)
    +\ip{  \B_\alpha (\pv\vq\cdot\vn +\tau(\pw u-\Pm u)), w}
    & =
    (f+e_{u_t}, w),
  \end{alignat*}
  for all $\vr \in {\bf V}_h$ and $w \in W_h$. Indeed,
the fact that $\Pm$ is the $L^2$-projection into $M_h$ was used in the third term of the left-hand  side of the first equation, and the
orthogonality property \eqref{eq:proj3} was used in the third term of the left-hand side of the second equation. To deal with that term, we also
used the fact that
\begin{equation}
\label{tauPM} \ip{\tau (P_Mu-u),\mu}=0 \Forall\mu \in M_h,
\end{equation}
given that, for each element $K\in\oh$, $\tau$ is constant on each face $e$ of $K$. Subtracting the equations
\eqref{eq:method-a} and \eqref{eq:method-b} from the
  above ones, respectively, we obtain equations \eqref{eq:err-a} and \eqref{eq:err-b}, respectively.

  The equation~\eqref{eq:err-e} follows directly from the
equation \eqref{eq:method-d}  and \eqref{heatequation-b}.

To prove  \eqref{eq:err-c}, we note that, by definition of $\eqhat$, \eqref{eq: comparison}, we have
 \begin{align*}
   &\ip{\B_\alpha \eqhat\cdot\vn,\mu}- \ip{\B_\alpha \eqhat\cdot\vn,\mu}_{\d\Omega}
\\
&\quad \quad \quad =[ \ip{ \vsig\cdot\vn,\mu}- \ip{ \vsig\cdot\vn,\mu}_{\d\Omega}]-
 [ \ip{ \sighat_h\cdot\vn,\mu}- \ip{\sighat_h\cdot\vn,\mu}_{\d\Omega}],
\end{align*}
since $P_M$ is the $L^2$-projection into $M_h$.  The first term of the right-hand side is equal to zero because  $\vsig$ is in $\Hdiv{\Omega}$
and the second because the normal component of $\sighat_h$ is
single valued by the equation \eqref{eq:method-c}. Hence, the identity \eqref{eq:err-c} holds.

 Next, let us prove \eqref{eq:err-f}. By
the equation~\eqref{eq:method-e} defining the HDG method, $u_h|_{t=0} = \pw u_0,$ and so $    \eu|_{t=0}
     = \pw u_0 -u_h|_{t=0}
    =  \pw u_0 -\pw u_0
    =  0.$
  It remains to prove the identity \eqref{eq:err-d2}. We have
\begin{alignat*}{2}
\eqhat\cdot\vn&=P_M (\vq\cdot\vn)-(\vq_h\cdot\vn+\tau\,(u_h- \uhat_h)) &&\quad\mbox{ by }\eqref{eq: comparison}~{\rm and}~\eqref{HDGtrace},
\\&=(\pv\vq\cdot\vn+\tau\,(\pw u-P_Mu))-(\vq_h\cdot\vn+\tau\,(u_h-\uhat_h))
&&\quad\mbox{ by }\eqref{eq:proj3},
\\&=\eq\cdot\vn  + \tau (\eu -\el)&&\quad\mbox{ by }\eqref{eq: comparison}.
\end{alignat*}
This completes the proof. $\qed$
\end{proof}

{ \bf 4.3 A first error bound}
\begin{lemma}
\label{coro:energy} For any $T\ge 0,$ we have
\[
\left(\|\eu(T)\|^2 +\|\eq\|^2_{\alpha}
 +2\trnorm{\sqrt{\tau}(\eu-\el)}^2_{\alpha}\right)^{1/2}
  \le  \|e_{u_t}\|_{L^1(L^2)} +d_\alpha\,\max_{t\in(0,T)}\|e_\vq\|_{\alpha,t}.
\]
\end{lemma}
\begin{proof}
   Taking {$\vr=\esig$ in \eqref{eq:err-a}, 
          $w=\eu$ in \eqref{eq:err-b},
          $\mu=-\B_\alpha\eqhat\cdot\vn$ in~\eqref{eq:err-e}
     and  $\mu=-\el$ in \eqref{eq:err-c}, and adding the resulting four}
  equations, we get
\[
 \frac12 \frac{d}{dt}\|\eu\|^2 +( \B_\alpha \eq , \eq) +\Psi_h = \; (e_\vq, \esig) + (e_{u_t},\eu),
\]
where by the definition of $\eqhat$, \eqref{eq:err-d2},
\begin{align*}
\Psi_h:=&-(  \eu, \dive \esig) + \ip{ \el,\esig\cdot\vn}
      -( \esig, \nabla \eu)_{\oh}
     +   \ip{ \B_\alpha \eqhat \cdot\vn, \eu
      - \el}
\\
      =&-\ip{  \eu, \esig\cdot\vn} \!+\!   \ip{  \el,\esig\cdot\vn}
         \! +\!   \ip{  \B_\alpha \eqhat \cdot\vn, \eu- \el }
\\
      =&\;\ip{(\B_\alpha \eqhat-\esig)\cdot\vn, \eu-\el}
=\;\ip{\B_\alpha (\sqrt{\tau}(\eu-\el)), \sqrt{\tau}(\eu-\el)}\,.
\end{align*}
 Integrating over the time interval $(0,T)$, and
using the fact that $\eu(0)=0$ by \eqref{eq:err-f},
 \begin{align*}
  \|\eu(T)\|^2 +2 \| \eq\|^2_{\alpha}
 &+2 \trnorm{  \sqrt{\tau}(\eu-\el)}^2_{\alpha} = 2\int_0^T\!\!( e_\vq, \esig)
+ 2\int_0^T\!\!(e_{u_t},\eu)\,.
\end{align*}
Since $2 \int_0^T( e_\vq,\B_\alpha \eq)  \le 2{ d_\alpha}  \| e_\vq\|_{\alpha} \| \eq\|_{\alpha}
 \le { d^2_\alpha}\| e_\vq\|^2_{\alpha} + \|\eq\|^2_{\alpha}$,
 by the property (iii) of Lemma
  \ref{appendix}, and since
$ \int_0^T( e_{u_t}, \eu)  \le  \int_0^T\|e_{u_t}\|\,\|\eu \|, $
\begin{alignat*}{1}
\|\eu(T)\|^2 +\| \eq\|^2_{\alpha}
 +2\trnorm{  \sqrt{\tau}(\eu-\el)}&_\alpha^2
\le { d_\alpha^2 \| e_\vq\|^2_{\alpha}} +2\int_0^T\|e_{u_t}\|\|\eu\|\quad {\rm for}~~ T>0\,.
\end{alignat*}
The result now easily follows from  Lemma \ref{differentialinequality} below with $A(t):= { d_\alpha^2 \| e_\vq\|^2_{\alpha,t}}$,
$B(t):=\|e_{u_t}(t)\|$ and with $E^2(t):=\|\eu(t)\|^2 +\| \eq\|^2_{\alpha,t}
 +2\,\trnorm{  \sqrt{\tau}(\eu-\el)}^2_{\alpha,t}\,. \quad\qed$
\end{proof}
\begin{lemma}[An integral inequality]
\label{differentialinequality} Suppose that, for any $t\ge0$,  we have that $ E^2(t)\le A(t)+2\,\int_{0}^t\,B(s)\,E(s)\,ds, $ for some
nonnegative functions $A$ and $B$. Then, for any $T>0$, $ E(T)\le \max_{t\in(0,T)} A^{1/2}(t)+\int_0^T\,B(s)\,ds. $
\end{lemma}

\begin{proof}
Setting $X(t)=\max_{t\in[0,T]} A(t)+2\,\int_0^t B(s)\,E(s)\,ds$, we see that, for $t\in(0,T)$, $\frac{d}{dt} X(t)= 2\,B(t)\, E(t) \le 2\,B(t)\,
\sqrt{X(t)}$, and so $\frac{d}{dt}\sqrt{X}(t)\le B(t)$. This implies that $\sqrt{X(t)}\le \sqrt{X}(0)+\int_{0}^t\,B(s)\,ds$, and the result
follows. $\quad\qed$ 
\end{proof}

{\bf 4.4 A second error bound} We derive next  an estimate of $\eq$ in the $L^\infty(0,T;L^2(\Omega))-$norm.
\begin{lemma}
\label{coro:energy2} Let $ S^2_{h}:=\langle{\tau}(\eu-\el), (\eu-\el)\rangle.$ For any $T >  0$, we have
\begin{alignat*}{1}
\big(\|\eq(T)\|^2 +S^2_{h}(T) +2\,\|\partial_t\eu\|^2_{-\alpha} \big)^{1/2} \le &\big(\|\eq(0)\|^2 +S^2_{h}(0)\big)^{1/2}
\\&+ d_\alpha\,\max_{t\in(0,T)}{ \|e_{u_t}\|_{-\alpha,t}}
+\|e_{\vq_t}\|_{L^1(L^2)}.
\end{alignat*}
\end{lemma}
{\em Proof.}
 By the adjoint property \eqref{eq: adjoint} and  the identity property \eqref{eq: right-inverse},
\begin{multline*}
2\int_0^T (\B_\alpha\eq, \I^*_{-\alpha}\partial_t\eq)  =2\int_0^T (\eq, \B^*_\alpha\I^*_{-\alpha}\partial_t\eq)\\=2\int_0^T (\eq,
\partial_t\eq)=\|\eq(T)\|^2-\|\eq(0)\|^2.\end{multline*}
Now, applying the operator $\I^*_{-\alpha}\partial_t$ to the first equation of the errors, \eqref{eq:err-a}, and taking $\vr:=\esig$, we obtain
\begin{multline*}
(\I^*_{-\alpha}\partial_t\eq,\esig) - (\I^*_{-\alpha}\partial_t\eu,\nabla\cdot\esig) \\ + \langle \I^*_{-\alpha}\partial_t\el,\esig\cdot{\bf
n}\rangle = (\I^*_{-\alpha}e_{\vq_t},\esig).
\end{multline*}
Integrating in time from $0$ to $T$ and using the identity of the previous step, we get
\begin{alignat*}{1}
\frac12\|\eq(T)\|^2 -\int_0^T(\I^*_{-\alpha}\partial_t\eu,\nabla\cdot\esig) &+\int_{0}^T\langle \I^*_{-\alpha}\partial_t\el,\esig\cdot{\bf
n}\rangle
\\&= \frac12\|\eq(0)\|^2+\int_0^T(\I^*_{-\alpha}e_{\vq_t},\esig).
\end{alignat*}
Now, taking $w:=\I^*_{-\alpha}\partial_t\eu$ in  equation \eqref{eq:err-b}, and integrating from $0$ to $T$,
\begin{multline*}
\int_0^T[(\partial_t\eu,\I^*_{-\alpha}\partial_t\eu) -(\esig,\nabla\I^*_{-\alpha}\partial_t\eu) \\ + \langle \B_\alpha \eqhat\cdot{\bf
n},\I^*_{-\alpha}\partial_t\eu\rangle] = \int_0^T(e_{u_t},\I^*_{-\alpha}\partial_t\eu). \end{multline*} Adding this equation to the one obtained
in the last step and, rearranging terms,
\begin{multline*}
\|\eq(T)\|^2+2\int_0^T(\partial_t\eu-e_{u_t},\I^*_{-\alpha}\partial_t\eu) +2\Phi_h \\ = \|\eq(0)\|^2+
2\int_0^T(\I^*_{-\alpha}e_{\vq_t},\esig)\\
= \|\eq(0)\|^2+ 2\int_0^T(e_{\vq_t},\eq), ~~\text{by  the  properties \eqref{eq: adjoint} and \eqref{eq: right-inverse},}
\end{multline*}
where
\begin{alignat*}{1}
\Phi_h:=&-\int_0^T(\I^*_{-\alpha}\partial_t\eu,\nabla\cdot\esig{})_{\oh} +\int_{0}^T\langle \I^*_{-\alpha}\partial_t\el,\esig\cdot{\bf n}\rangle
\\
&- \int_0^T(\esig,\nabla\I^*_{-\alpha}\partial_t\eu)_{\oh} + \int_{0}^T \langle \B_\alpha \eqhat\cdot{\bf n},\I^*_{-\alpha}\partial_t\eu\rangle
\\
=&\int_0^T[-\langle \I^*_{-\alpha}\partial_t\eu,\esig\cdot{\bf n}\rangle \!+\!\langle \I^*_{-\alpha}\partial_t\el,\esig\cdot{\bf n}\rangle
\!+\!\langle \B_\alpha \eqhat\cdot{\bf n},\I^*_{-\alpha}\partial_t\eu\rangle]
\\
=&\int_{0}^T [\langle \B_\alpha( \eqhat-\eq)\cdot{\bf n},\I^*_{-\alpha}\partial_t(\eu-\el)\rangle + \langle \B_\alpha \eqhat\cdot{\bf
n},\I^*_{-\alpha}\partial_t\el\rangle]
\\
=&\int_{0}^T \langle \B_\alpha(\eqhat-\eq)\cdot{\bf n},\I^*_{-\alpha}\partial_t(\eu-\el)\rangle
\hskip.9truecm\mbox{ by equations \eqref{eq:err-e} and \eqref{eq:err-c},} \\
=&\int_{0}^T \langle (\eqhat-\eq)\cdot{\bf n},\B_\alpha^*\I^*_{-\alpha}\partial_t(\eu-\el)\rangle
\hskip.9truecm\mbox{ by  the adjoint property \eqref{eq: adjoint},}\\
=&\int_{0}^T \langle \tau(\eu-\el),\partial_t(\eu-\el)\rangle=\frac12 S^2_{h}(T)-\frac12 S^2_{h}(0)
\end{alignat*}
 by the identity property \eqref{eq: right-inverse} and the error equation \eqref{eq:err-d2}.
Therefore, for any $T>0$,
\[
  \|\eq(T)\|^2 + S^2_{h}(T)  +2\|\partial_t\eu\|^2_{-\alpha}  = \|\eq(0)\|^2 +  S^2_{h}(0) + 2\int_0^T[(e_{\vq_t}, \eq) +
(e_{u_t},\I^*_{-\alpha}
\partial _t\eu)]\,.
\]
But,  by property (iv) of Lemma \ref{appendix},
\begin{alignat*}{1}
2\int_0^T(e_{u_t},\I^*_{-\alpha} \partial _t\eu) \le &\, d_\alpha^2  \|e_{u_t}\|_{-\alpha}^2 +\|\partial_t\eu\|^2_{-\alpha},
\end{alignat*}
and since $ \int_0^T(e_{\vq_t}, \eq) \le \int_0^T\|e_{\vq_t}\|\, \|\eq\|, $ we have, that, for any $T>0$, \[ \|\eq(T)\|^2 +S^2_{h}(T)
+2\,\|\partial_t\eu\|^2_{-\alpha} \le \|\eq(0)\|^2 + S^2_{h}(0) + d_\alpha^2  \|e_{u_t}\|_{-\alpha}^2 + 2\,\int_0^T\|e_{\vq_t}\| \,\|\eq\| .
\]
Finally, the desired inequality follows from  Lemma \ref{differentialinequality} with $B(t):=\|e_{\vq_t}(t)\|$ and
\begin{alignat*}{1}
A(t)&:= \|\eq(0)\|^2 + S^2_{h}(0) + d_\alpha^2\,\max_{t\in(0,T)}  { \|e_{u_t}\|_{-\alpha,t}^2},
\\
E^2(t)&:= \|\eq(t)\|^2 +S^2_{h}(t) +2\,\|\partial_t\eu\|^2_{-\alpha,t}\,.\quad\Box
\end{alignat*}

We still need to estimate the  term $ \|\eq(0)\|^2 +S^2_h(0)$   in Lemma  \ref{coro:energy2}.
\begin{lemma}
\label{lemma:time0} We have that $
 \|\eq(0)\|^2
+S^2_h(0) \le \frac{d_\alpha^2}{c_\alpha\,\Gamma(\alpha+2)}\|e_\vq(0)\|^2$, provided $e_{u_t}\in\mathcal{C}^{0}(0,\epsilon;L^2(\Omega))$ and
$e_\vq\in\mathcal{C}^{1}(0,\epsilon; {\bf L}_2(\Omega))$ for some positive $\epsilon.$
\end{lemma}
\begin{proof}
Setting $\Theta_h(t):= \|\eq(t)\|^2 +S_h^2(t)$, we get, by the coercivity property (i) of Lemma \ref{appendix}, that
\[
\big(c_\alpha t^\alpha \int_0^t \Theta_h \big)^{1/2} \le \big(\|\eq\|_{\alpha,t}^2 +\trnorm{\sqrt{\tau}(\eu-\el)}^2_{\alpha,t}\big)^{1/2} \le
\int_0^t\|e_{u_t}\| +d_\alpha\underset{t^*\in(0,t)}{\max}\|e_\vq\|_{\alpha,t^*}
\]
by Lemma  \ref{coro:energy}. Then
\[
\Theta_h^{1/2}(0) = \lim_{t\downarrow0}\, t^{-(1+\alpha)/2}\,\big(t^\alpha\int_0^t\Theta_h(s)\,ds\big)^{1/2} \le
c^{-1/2}_\alpha\,(T_1+d_\alpha\,T_2),
\]
where $ T_1:= \lim_{t\downarrow0}\,t^{-(1+\alpha)/2}\int_0^t\|e_{u_t}\|=0$, by the assumption on $e_{u_t}$,
\begin{alignat*}{1}
T_2:=&\; \lim_{t\downarrow0}\,t^{-(1+\alpha)/2}\underset{t^*\in(0,t)}{\max}\|e_\vq\|_{\alpha,t^*} =\;
\frac{1}{\Gamma^{1/2}(\alpha+2)}\,\|e_\vq(0)\|,
\end{alignat*}
by property (v) of Lemma \ref{appendix}. This completes the proof. $\quad \qed$
\end{proof}

{\bf 4.5 The error estimates} We are now ready to obtain our  HDG error estimates. By Lemmas  \ref{coro:energy}, \ref{coro:energy2}, and
\ref{lemma:time0}, we get
\[
\|(u-u_h)(T)\|\le \|e_u(T)\| +\trnorm{[e_\vq,e_u]}_{1,\alpha}~~{\rm and}~~ \|(\vq-\vq_h)(T)\|\le \|e_\vq(T)\| +\trnorm{[e_\vq,e_u]}_{2,\alpha},
\]
where
\begin{alignat*}{1}
\trnorm{[\vq,u]}_{1,\alpha}&:=\|u_t\|_{L^1(L^2)}+d_\alpha\,\max_{t\in(0,T)} { \|\vq\|_{\alpha,t}},
\\
\trnorm{[\vq,u]}_{2,\alpha}&:= \frac{d_\alpha}{c^{1/2}_\alpha\,\Gamma^{1/2}(\alpha+2)} \|\vq(0)\|+\|\vq_t\|_{L^1(L^2)}+d_\alpha\,
\max_{t\in(0,T)}{ \|u_t\|_{-\alpha,t}}.
\end{alignat*}
Note that when $\alpha=0$, we recover the error estimates for the HDG methods for the heat equation of \cite[Theorem 2.1]{ChabaudCockburn12}
since in this case $d_0=1$, $c_0=1$ and $\Gamma(2)=1$. If we now use the approximation properties of the projections $\pv$ and $\pw$ of Theorem
\ref{thm:proj}, we obtain our optimal HDG error estimates.

\begin{theorem}
\label{thm:ee} Assume that $u\in\mathcal{C}^1(0,T; H^{k+1}(\Omega))$ and $\vq\in\mathcal{C}^1(0,T;{\bf H}^{k+1}(\Omega))$. Assume also that
$\tau^*_K$ and $1/\tau^{\max}_K$ are bounded by $\mathsf {C}$. Then we have that
\[
\|(u-u_h)(T)\|\le\, C_1\,h^{k+1}
\quad\mbox{ and } \qquad
\|(\vq-\vq_h)(T)\|\le\, C_2\,h^{k+1}.
\]
The constant $C_i$, $i=1,2$, only depends on ${C}$,  $\alpha$,
$\|u\|_{\mathcal{C}^1(H^{k+1})}$, and on $\|\vq\|_{\mathcal{C}^1(H^{k+1})}$.
\end{theorem}

Note that, provided that the exact solution is smooth, the above error estimates are uniform for $\alpha\in [\alpha^*, 0]$ provided
$\alpha^*>-1$. This is not true for $\alpha^*=-1$ since the coefficients $d_\alpha$ and $1/c_\alpha$ behave like $1/(\alpha+1)$ as $\alpha$ goes
to $-1$.  Note also that these results hold even when the domain
  $\Omega$ is not convex.

\section{Superconvergence and post-processing}\label{sec:super}
In this section, we carry out the second part of our a priori error analysis. We prove superconvergence results which will allow us to compute a
new, better approximation to $u$ by means of an element-by-element postprocessing. We begin by describing such approximation. Then, we show how
to get our superconvergence result by a duality argument.

Following \cite{GastaldiNochetto89,Stenberg88,Stenberg91,ChabaudCockburn12}, for each fixed $t\in[0,T]$, we define  the postprocessed HDG
approximation $u_{h}^\star \in \pol_{k+1}(K)$ to $u$ for  each
 simplex $K\in\oh$, as follows:
\begin{subequations}
\label{eq:ustar}
  \begin{alignat}{2}
\label{eq:ustar-a}
 (u^\star_h(t) , 1 )_K
     =&\;(u_h(t), 1 )_K
  \\
\label{eq:ustar-b}
    (\nabla u^\star_h(t) , \nabla w )_K
     =&-(\vq_h(t) , \nabla w )_K
    &&\qquad\text{ for all } w \in \pol_{k+1}(K).
  \end{alignat}
\end{subequations}
It is not difficult to obtain the following result:
\begin{equation}
\label{ustar}  \|u(t)-u^\star_h(t)\|_K  \le C\,h_{{K}}^{k+2}\,|u(t)|_{H^{k+2}(K)}+\|P_0 \eu(t)\|_K+C\,h\,\|\eq(t)\|_K.
\end{equation}
Here $P_0$ is the $L^2(\Omega)$-projection into the space of functions which are constant on each element $K\in\oh$.\\

{\bf 5.1 A first estimate of $\|P_0\eu(T)\|$ by duality argument} We see that if the term $\|{P_0\eu}\|$  is of order $O(h^{k+2})$, we would
have that the postprocessed approximation $u_h^\star$ would converge faster than the original approximation $u_h$. To obtain such an estimate,
the traditional duality approach consists in, since we can write $\|P_0\eu(T)\|= \sup_{{\Theta}\in
  C^\infty_0(\Omega)}\frac{(P_0\eu(T),\Theta)}{\|\Theta\|}$,
estimating the expression $(P_0\eu(T),\Theta)$ by using the solution of the dual problem
  \begin{subequations}
    \label{eq:dual}
    \begin{align}
      \label{eq:dual-1}
     \vPhi + \nabla \vPsi & = 0 &&\text{ on } \Omega \times (0,T), \\
      \label{eq:dual-2}
      \vPsi_t-\dive \B_\alpha^* \vPhi & = 0 &&\text{ on } \Omega \times (0,T), \\
      \label{eq:dual-3}
      \vPsi &=0  &&\text{ on } \d \Omega \times (0,T), \\
      \label{eq:dual-4}
      \vPsi(T) &=\Theta &&\text{ on } \Omega.
    \end{align}
  \end{subequations}

In the next result, we give an expression for the quantity $(P_{0}\eu(T),\vTheta)$ in terms of the errors $\esig$, $\eu$ and the solution of the
dual problem. In it, $\mathrm{I}_h$ is any interpolation operator from $L^2(\Omega)$ into $W_h\cap H^1_0(\Omega)$, $\Pw$ is the $L^2$-projection
into $W_h$ and $\boldsymbol \Pi^{\mbox{{\rm\tiny{BDM}}}}$ is the well-known projection associated to the lowest-order Brezzi-Douglas-Marini
(BDM) space, see \cite{BrezziFortin91}.

\begin{lemma}
\label{lemma:duality3}  Assume that $k\ge1$. Then, for any $T>0$,  $(P_{0}\eu(T),\vTheta)$  equals
\begin{align*}
\; \int_0^{T} [(\eq,\B^*_\alpha(-\boldsymbol \Pi^{\mbox{{\rm\tiny{BDM}}}}\nabla\vPsi +\nabla\mathrm{I}_h\vPsi)) +(e_{\vq}
,&\B^*_\alpha(\boldsymbol \Pi^{\mbox{{\rm\tiny{BDM}}}}\nabla\vPsi -\nabla \Pw\vPsi))
\\&\quad\quad
+(\partial_t \eu-e_{u_t},P_{0}\vPsi-\mathrm{I}_h\vPsi) ].
\end{align*}
\end{lemma}
\begin{proof}
Since $\vPsi(T)=\vTheta$ by \eqref{eq:dual-4} and $\eu(0)=0$ by \eqref{eq:err-f}, we have
\begin{multline*}
(P_0\eu(T),\vTheta)
  = \int_0^{T} [(\partial_t P_0\eu,\vPsi)+(P_0\eu,\vPsi_t)]
  \\= \int_0^{T} [(\partial_t \eu,P_0\vPsi)+(\eu, P_0\nabla\cdot \B_\alpha^* \vPhi)]
\end{multline*}
by the definition of the $L^2$-projection $P_0$ and by  \eqref{eq:dual-2}.

Let us work on the last term of the right-hand side. By the commutativity property $P_0\nabla\cdot=\nabla\cdot\boldsymbol
\Pi^{\mbox{{\rm\tiny{BDM}}}}$, we have $ (\eu, P_0\nabla\cdot \B_\alpha^* \vPhi)=(\eu, \nabla\cdot\B_\alpha^*\boldsymbol
\Pi^{\mbox{{\rm\tiny{BDM}}}} \vPhi).$  Since $k\ge1$, we can take ${\bf r}:=\B_\alpha^*\boldsymbol \Pi^{\mbox{{\rm\tiny{BDM}}}} \vPhi$ in the
first error equation {\eqref{eq:err-a}}, to get
\begin{alignat*}{2}
 (\eu, P_0\nabla\cdot \B_\alpha^* \vPhi)\!\!=& (\eu, \nabla\cdot\B_\alpha^*\boldsymbol \Pi^{\mbox{{\rm\tiny{BDM}}}} \vPhi),\\
\!\! =&(\eq,\B_\alpha^*\boldsymbol \Pi^{\mbox{{\rm\tiny{BDM}}}} \vPhi) \!\!+\!\!\ip{\el,\B_\alpha^*\boldsymbol \Pi^{\mbox{{\rm\tiny{BDM}}}}
\vPhi\cdot{\bf n}}
-(e_\vq,\B_\alpha^*\boldsymbol \Pi^{\mbox{{\rm\tiny{BDM}}}} \vPhi)\\
\!\!=& (\eq,\B_\alpha^*\boldsymbol \Pi^{\mbox{{\rm\tiny{BDM}}}} \vPhi)-(e_\vq,\B_\alpha^*\boldsymbol \Pi^{\mbox{{\rm\tiny{BDM}}}} \vPhi),
\end{alignat*}
since $\ip{\el,\B_\alpha^*\boldsymbol \Pi^{\mbox{{\rm\tiny{BDM}}}} \vPhi\cdot{\bf n}} \!\!=\!\!\ip{\el,\B_\alpha^*\boldsymbol
\Pi^{\mbox{{\rm\tiny{BDM}}}} \vPhi\cdot{\bf n}}_{\partial\Omega}=0$ because $\B_\alpha^*\boldsymbol \Pi^{\mbox{{\rm\tiny{BDM}}}} \vPhi\in
\Hdiv\Omega$ and $\el=0$ on $\partial\Omega$ by \eqref{eq:err-e}\,.

Integrating in time from $0$ to $T$ and using the adjoint property \eqref{eq: adjoint}, we get
\[
\int_0^T(\eq,\B_\alpha^*(\boldsymbol \Pi^{\mbox{{\rm\tiny{BDM}}}}\vPhi)) =\int_0^T(\eq,\B_\alpha^*(-\boldsymbol
\Pi^{\mbox{{\rm\tiny{BDM}}}}\nabla\vPsi +\nabla\mathrm{I}_h\vPsi)) -\int_0^T(\esig,\nabla\mathrm{I}_h\vPsi).
\]
But, by the error equation \eqref{eq:err-b} with
$w:=\mathrm{I}_h\vPsi$,
\begin{alignat*}{1}
(\esig,\nabla\mathrm{I}_h\vPsi) &=(\partial_t\eu-e_{u_t},\mathrm{I}_h\vPsi) -\ip{\B_\alpha \eqhat
  \cdot{\bf n},\mathrm{I}_h\vPsi}
=(\partial_t\eu -e_{u_t},\mathrm{I}_h\vPsi)
\end{alignat*}
since
 $\ip{\B_\alpha \eqhat\cdot{\bf n},\mathrm{I}_h\vPsi}= \ip{\B_\alpha \eqhat\cdot{\bf n},\mathrm{I}_h\vPsi}_{\partial\Omega}=0$ because the
normal component of $\B_\alpha \eqhat$ is single valued by \eqref{eq:err-c} and $\mathrm{I}_h\vPsi=0$ on $\partial\Omega$ by the boundary
condition  \eqref{eq:dual-3}.

Then, putting together all the above intermediate steps,  $(P_0\eu(T),\vTheta)$ equals
\[
 \int_0^{T} [(\eq,\B_\alpha^*(\nabla\mathrm{I}_h\vPsi-\boldsymbol \Pi^{\mbox{{\rm\tiny{BDM}}}}\nabla\vPsi)) -(e_{\vq}
,\B_\alpha^*\boldsymbol \Pi^{\mbox{{\rm\tiny{BDM}}}} \vPhi) +(\partial_t \eu,P_0\vPsi-\mathrm{I}_h\vPsi) +(e_{u_t},\mathrm{I}_h\vPsi)].
\]
 Therefore, the  desired result now follows after noting that
\[
\int_0^T(e_{\vq} ,\B_\alpha^*\boldsymbol \Pi^{\mbox{{\rm\tiny{BDM}}}} \vPhi) = \int_0^T(e_\vq,\B_\alpha^*(\boldsymbol
\Pi^{\mbox{{\rm\tiny{BDM}}}}\nabla\vPsi-\nabla \Pw\vPsi)),
\]
and that $(e_{u_t},\mathrm{I}_h\vPsi) =(e_{u_t},\mathrm{I}_h\vPsi-P_0\vPsi)$, by \eqref{eq:dual-1}, the definition of $\Pw\vPsi$ and  the
orthogonality property of the projection $\pv$, \eqref{eq:proj1 new}; and by the definition of $P_0\vPsi$ and the orthogonality property of the
projection $\pw$, \eqref{eq:proj2}. $\quad \qed$
\end{proof}

 Now,  as a direct consequence of  the
previous lemma and by property (ii) of Lemma \ref{appendix}, we have that
\begin{align*}
\big|(P_{0}\eu(T),\vTheta) \big| \le& \|\eq\|_{L^{\infty}(L^2)}\,\|\B^*_\alpha(\boldsymbol \Pi^{\mbox{{\rm\tiny{BDM}}}}\nabla\vPsi
-\nabla\mathrm{I}_h\vPsi)\|_{L^{1}(L^2)}
\\
& +\|e_\vq\|_{L^{\infty}(L^2)}\, \|\B_\alpha^*(\boldsymbol \Pi^{\mbox{{\rm\tiny{BDM}}}}\nabla\vPsi -\nabla
\Pw\vPsi)\|_{L^{1}(L^2(\mathcal{T}_h))}
\\& +(\|\partial_t \eu\|_{-\alpha} +\|e_{u_t}\|_{-\alpha})\, \|\mathrm{I}_h\vPsi-P_{0}\vPsi\|_{\alpha}.
\end{align*}
This implies the following estimate of $\|P_0\eu(T)\|$;
\begin{equation}\label{firstestimate}
\|P_0\eu(T)\| \le \; \mathsf{H}_1(\vTheta)\,(\|\eq\|_{L^\infty(L^2)}
    + \|e_\vq\|_{L^\infty(L^2)})
 + \mathsf{H}_2(\vTheta)\left(\|\partial_t \eu\|_{-\alpha}
    + \|e_{u_t}\|_{-\alpha}\right),
\end{equation}
where
{\small
\begin{align*}
\mathsf{H}_1(\vTheta):=&\!\!\!\sup_{{\Theta\in \mathcal{C}^\infty_0(\Omega)}} \!\!\! \max\{\frac{\|\B_\alpha^*(\boldsymbol
\Pi^{\mbox{{\rm\tiny{BDM}}}}\nabla\vPsi \!-\!\nabla\mathrm{I}_h\vPsi)\|_{L^1\!(\!L^2\!)}}{\|{\Theta}\|}, \frac{\|\B_\alpha^*(\boldsymbol
\Pi^{\mbox{{\rm\tiny{BDM}}}}\nabla\vPsi \!-\!\nabla
  \Pw\vPsi)\|_{L^1\!(\!L^2(\mathcal{T}_h)\!)}}{\|{\Theta}\|}\},
\\
\mathsf{H}_2(\vTheta):=&\sup_{{\Theta\in \mathcal{C}^\infty_0(\Omega)}} \frac{\|P_0\vPsi-\mathrm{I}_h\vPsi\|_{\alpha}}{\|{\Theta}\|}.
\end{align*}
}
The quantity $\mathsf{H}_1(\vTheta)$ can be bounded by
\[
C\,h\, \sup_{{\Theta\in \mathcal{C}^\infty_0(\Omega)}} \frac{ \|\B_\alpha^*\vPsi\|_{L^1(H^2)}}{\|{\Theta}\|} \le C\,h\, \sup_{{\Theta\in \mathcal{C}^\infty_0(\Omega)}} \frac{
  \|B_\alpha^*\Delta\vPsi\|_{L^1(L^2)}}{\|{\Theta}\|}
= C\,h\, \sup_{{\Theta\in \mathcal{C}^\infty_0(\Omega)}} \frac{ \|\vPsi_t\|_{L^1(L^2)}}{\|{\Theta}\|},
\]
where, to get the inequality, we used the well-known elliptic regularity
property
\begin{alignat}{1}
\label{ellipticregularity}
\|v\|_{H^2(\Omega)}\le C\,\|\Delta v\|
\quad\mbox{ for any }v\in
H^1_0(\Omega)\cap H^2(\Omega),
\end{alignat}
which holds for convex polyhedral domains.

The quantity $\mathsf{H}_2(\vTheta)$ can be bounded by
\[
Ch \sup_{{\Theta\in \mathcal{C}^\infty_0(\Omega)}} \frac{1}{\|{\Theta}\|} \bigg(\int_0^T\|\nabla \vPsi\|\,\|\B_\alpha^*\nabla\vPsi\|\bigg)^{1/2}.
\]

Our next task is to obtain estimates of $\int_0^T\|\vPsi_t\|$ and $\int_0^T\|\nabla \vPsi\|\|\B_\alpha^*\nabla\vPsi\|$.\\

{\bf 5.2 A priori estimates for the dual solution} The estimates we need are gathered in the following result.
\begin{lemma}
\label{prop:parabolicregularity} For any $\Theta\in H^1_0(\Omega)$ and any $\delta\in(0,T)$, we have that
\begin{alignat*}{2}
\int_0^T\| \vPsi_t \| \le &\; \frac{C}{\alpha+1}\,\big( \sqrt{\ell(\delta)}\|\Theta\|+ \delta^{(\alpha+1)/2}\|\nabla\Theta\|\big),
\\
\int_0^T\| \nabla \vPsi \|\,\|\B_\alpha^*\nabla \vPsi\| \le &\; C\,\|\Theta\|\left(\ell(\delta) \|\Theta\|
+\,\frac{\delta^{(\alpha+1)/2}}{\alpha+1} \,\|\nabla\Theta\|\right),
\end{alignat*}
where $\ell(\delta)=\log(T/\delta).$ The constant $C$ is independent of $\Psi, T$ and $\alpha$.
\end{lemma}
\begin{proof}
First, we define the auxiliary function $v$: for each time $t\in [0,T]$,
\[
\Delta v(t):=\mathcal{R}\Psi(t) \quad\mbox{ in }\Omega \quad\mbox{ and }\quad v(t)|_{\partial \Omega}=0,
\]
 where $\mathcal{R}$ is the time-reversal operator for the interval $[0,T]$, that is, $\mathcal{R}\psi(t)=\psi(T-t)$.
For the moment, we assume the following properties of the function $v$:
\begin{alignat}{1}
 t^{(1-\alpha)/2}\| \Delta v_t(t)\|+\|\nabla(\Delta v(t))\| &\le C\,\min\{t^{-(\alpha+1)/2} \,\|\Theta\|,
\,\|\nabla\Theta\|\}\label{eq: estimate1},
\\
t^{-\alpha}\|v_t(t)\| + \|\Delta v(t)\|&\le C\,\|\Theta\|,\label{eq: estimate2}
\\
\label{eq: Estimate for t Delta v_t^2}\int_0^T t\|\Delta v_t\|^2\,dt & \le \frac{C}{(1+\alpha)^2}\|\Theta\|^2\,.
\end{alignat}
Using the relation $\Delta v(t)=\mathcal{R}\Psi(t)$ and the above  inequalities, we obtain
\begin{alignat*}{1} (T-t)^{(1-\alpha)/2}\| \vPsi_t(t)\| +\|\nabla\vPsi (t)\|
&\le C\,\min\{(T-t)^{-(\alpha+1)/2}\,\|\Theta\|, \|\nabla\Theta\|\}, \\
\|\B_\alpha^*\vPsi(t)\| &\le C\,(T-t)^\alpha\,\|\Theta\|
\\
 \int_0^T (T-t)\|\vPsi_t\|^2\,dt&\le  \frac{C}{(1+\alpha)^2}\|\Theta\|^2.
\end{alignat*}
This implies
\begin{alignat*}{1}
\|\B_\alpha^* \nabla\vPsi\|^2 &= -(\B_\alpha^*\Delta\vPsi, \B_\alpha^*\vPsi) = (\vPsi_t, \B_\alpha^*\vPsi) \le \|\vPsi_t\|\,\|
\B_\alpha^*\vPsi\| \le C\,(T-t)^{\alpha-1}\,\|\Theta\|^2,
\end{alignat*}
and so $ \|\nabla\vPsi\|\,\|\B_\alpha^* \nabla\vPsi\| \le
C\,\min\{(T-t)^{-1}\,\|\Theta\|^2,(T-t)^{(\alpha-1)/2}\,\|\Theta\|\,\|\nabla\Theta\|\}.$ Hence
\begin{alignat*}{1}
\int_0^T \|\vPsi_t\| &\le \int_0^{T-\delta} \|\vPsi_t\| +\int_{T-\delta}^T \|\vPsi_t\|
\\
&\le \sqrt{\log(T/\delta)} \left(\int_0^{T-\delta} (T-t)\,\|\vPsi_t\|^2\right)^{1/2} +C\,\int_{T-\delta}^T
(T-t)^{(\alpha-1)/2}\,\|\nabla\Theta\|
\\
&\le C\,\sqrt{\log(T/\delta)}\frac{\|\Theta\|}{\alpha+1} +C\,\frac{\delta^{(\alpha+1)/2}}{\alpha+1}\|\nabla\Theta\|,
\end{alignat*}
and
\begin{alignat*}{1}
\int_0^T \|\nabla\vPsi\|\,\|\B_\alpha^*\nabla\vPsi\| &\le \int_0^{T-\delta} \|\nabla\vPsi\|\,\|\B_\alpha^*\nabla\vPsi\| +\int_{T-\delta}^T
\|\nabla\vPsi\|\,\|\B_\alpha^*\nabla\vPsi\|
\\
&\le\,C\, \int_0^{T-\delta} (T-t)^{-1}\,\|\Theta\|^2 +C\,\int_{T-\delta}^T (T-t)^{(\alpha-1)/2}\,\|\Theta\|\,\|\nabla\Theta\|
\\
&\le C\,\log(T/\delta)\|\Theta\|^2 +C\,\frac{\delta^{(\alpha+1)/2}}{\alpha+1} \,\|\Theta\|\,\|\nabla\Theta\|\,.
\end{alignat*}
Therefore, the remaining task is to show the inequalities \eqref{eq: estimate1}, \eqref{eq: estimate2}, and \eqref{eq: Estimate for t Delta
v_t^2}. Using the fact that $\mathcal{R}\partial_t=-\partial_t\mathcal{R}$ and that $\mathcal{R}\B_\alpha^*=\B_\alpha\mathcal{R}$, we see that
\begin{equation*}
v_t-\B_\alpha \Delta v(t)=0 \;\mbox{ in }\Omega\times(0,T), \quad v=0\;\mbox{ on } {\partial \Omega}\times(0,T), ~\mbox{ and }\;
v(0)=\Delta^{-1}\Theta.
\end{equation*}
Thus, by \cite[Theorems 4.1 and 4.2]{McLean2010}, \eqref{eq: estimate1} and  \eqref{eq: estimate2} immediately follow. To prove inequality
 \eqref{eq: Estimate for t Delta v_t^2}, we use the identity
\[
\int_0^T t\|\Delta v_t\|^2\,dt=  t(\Delta v_t(t),\Delta v(t))\bigg|_0^T-\frac12 \|\Delta v(t)\|^2\bigg|_0^T -\int_0^T t(\Delta v_{tt},\Delta
v)\,dt,
\]
and the inequalities \eqref{eq: estimate1} and \eqref{eq: estimate2}, to get
\[\int_0^T t\|\Delta v_t\|^2\,dt
\le C\,\|\Theta\|^2  + \int_0^T |t(\Delta v_{tt},\Delta v)|\,dt.
\]
It remains to estimate the second term of the right-hand side. To do that, we first note that, since  the operator $-\Delta$ (with homogeneous
Dirichlet boundary conditions) has  a complete orthonormal eigensystem $\{\lambda_m,\phi_m\}_{m=1}^\infty$ ($\phi_m  \in H^1_0(\Omega)$ and
$0<\lambda_1\le\lambda_2\le\lambda_3\le\cdots$), one may show that the solution $v$  is given by the Duhamel formula
\begin{equation*}
 v(t)=\sum_{m=1}^\infty
        E_{\mu}(-\lambda_m t^{\mu})(v(0),\phi_m)\phi_m\quad{\rm with}~~\mu=\alpha+1,
\end{equation*}
where $E_{\mu}(t):=\sum_{p=0}^\infty\frac{t^p}{\Gamma(\mu p+1)}$, is the
 Mittag-Leffler function; see \cite{McLeanMustapha2007}. Thus,
 \begin{align*}
t(\Delta\,v_{tt},\Delta v)&= \sum_{m=1}^\infty G_{m,\mu}(t)
        \,(\Delta v(0),\phi_m)^2=\sum_{m=1}^\infty G_{m,\mu}(t)
        \,(\Theta ,\phi_m)^2,
\end{align*}
where $G_{m,\mu}(t):=t\,E_{\mu}(-\lambda_m t^\mu)\frac{d^2}{dt^2}(E_{\mu}(-\lambda_m t^\mu))$\,. Since, by the proof of Theorems 4.1 and 4.2 in
\cite{McLean2010}, we have that $ |G_{m,\mu}(t)| \le C\,\min\{\lambda_mt^{\mu-1},\lambda_m^{-2}t^{-2\mu-1}\}, $ we get
\begin{alignat*}{1}
\int_0^T |G_{m,\alpha}(t)|\,dt & \le C\lambda_m\int_0^{\lambda_m^{-1/\mu}}
t^{\mu-1}\,dt+C\lambda_m^{-2}\int_{\lambda_m^{-1/\mu}}^Tt^{-2\mu-1}\,dt\le
  \frac{C}{(\alpha+1)^2},
  \end{alignat*}
and therefore,
\[
\int_0^T |t(\Delta v_{tt},\Delta v)|\,dt \le \sum_{m=1}^\infty \int_0^T |G_{m,\mu}(t)|\,dt
        \,(\Theta ,\phi_m)^2 \le
\frac{C}{(\alpha+1)^2}\|\Theta\|^2\,.
\]
This completes the proof. $\quad \qed$
\end{proof}

{\bf 5.3 Compensating for the lack of regularity of $\Theta$} Note that the a priori estimates of Lemma \ref{prop:parabolicregularity} do use
the $H^1_0(\Omega)-$seminorm of $\Theta$ whereas the bounds of the quantities $\mathsf{H}_i(\Theta)$ can only use its $L^2(\Omega)-$norm. To
remedy this lack of regularity, we take advantage of the fact that $P_0\eu(T)$ lies in a finite dimensional space.

Let $\mathcal{T}_{h'}$ be a triangulation of $\Omega$ obtained by refining each of the simplexes of the triangulation $\mathcal{T}_{h}$, and let
$W^c_{h'}$ be the space of {\em continuous} functions which are polynomials of degree $k$ on each element of $\mathcal{T}_{h'}$. Finally let
$\mathrm{P}_{h'}$ be the $L^2$-projection from $W_h$ to $W^c_{h'}$. Then, we have the following result.

\begin{lemma} [{\cite[Appendix A.3]{ChabaudCockburn12}}]
\label{lemma:aux} For any triangulation $\mathcal{T}_{h}$ of $\Omega$, we can always find a refinement $\mathcal{T}_{h'}$ for which we have
\begin{alignat*}{2}
\|\nabla \mathrm{P}_{h'}\theta\| \le \frac{C_{k,d}}{\rho}\, \|\theta\|\quad\forall~ \theta\in W_h, \quad\mbox{ and } \quad \| \varepsilon \|\le
2\,\sup_{\theta\in W_h}\frac{( \varepsilon ,\mathrm{P}_{h'}\theta)}{\|\theta\|}\quad\forall~ \varepsilon \in W_h.
\end{alignat*}
Here the constant $C_{k,d}$ depends solely on the polynomial degree $k$ and the dimension $d$ of the spacial domain $\Omega$, and  ${\rho:=\min_{K\in\oh} \rho_K}$
where $\rho_K$ denotes the radius of the {largest} ball included in the simplex $K$.
\end{lemma}

Roughly speaking, the second inequality gives us an alternative manner to estimate the $L^2(\Omega)$-norm of $\varepsilon:=P_0\eu(T)$. Indeed,
it allows us to take $\Theta$ of the form $P_h'\theta$ only. The first inequality takes care of the lack of smoothness of $\Theta$ but at the
price of the appearance of the factor $\rho$ in the denominator. We can now
 modify the a priori inequalities of Lemma
\ref{prop:parabolicregularity} as follows.

\begin{lemma}
\label{coro:parabolicregularity} Let $(\vPhi,\vPsi)$ be the solution of the dual problem with $\vTheta:=\mathrm{P}_{h'}\theta$ where $\theta\in
W_h$ and $\mathrm{P}_{h'}$ satisfies Lemma \ref{lemma:aux}. Then
\begin{alignat*}{1}
\int_0^T\| \vPsi_t \| \le &\;\frac{C}{\alpha+1}\,\sqrt{\log{\kappa}}\; \|\theta\|\quad{\rm and}\quad \int_0^T\| \nabla\vPsi
\|\|\B_\alpha^*\nabla \vPsi\| \le \;\frac{C}{\alpha+1}\,
 {\log{\kappa}}\,\|\theta\|^2,
\end{alignat*}
where,  $\kappa > 1$ is the solution of $  \kappa^{\alpha+1}\log\kappa= C_{k,d}^2\,T^{\alpha+1}/\rho^2.$  Here ${\rho:=\min_{K\in\oh} \rho_K}$
and $\rho_K$ denotes the radius of the {largest} ball included in the simplex $K$.
\end{lemma}

\begin{proof} We prove the first estimate; the proof of the second is almost
  identical. From the first inequality of  Lemma  \ref{prop:parabolicregularity}
  with $\vTheta:=\mathrm{P}_{h'}\theta$, the fact that $\mathrm{P}_{h'}$
is an $L^2$-projection, and the first inequality of Lemma \ref{lemma:aux}, we obtain
\begin{alignat*}{1}
\int_0^T\| \vPsi_t \| &\le\; \frac{C}{\alpha+1}\,\big( \sqrt{\log({T}/\delta)}+ \delta^{(\alpha+1)/2}\frac{C_{k,d}}{\rho}\,\big)\|\theta\|
= \frac{2\,C}{\alpha+1}\, \sqrt{\log(\kappa)}\|\theta\|,
\end{alignat*}
if we take $\delta:=T/\kappa$ and use the definition of $\kappa$. This completes the proof. $\quad\qed$
\end{proof}

{\bf 5.4 The  estimate of the postprocessed approximation} We can now insert the estimates of the previous corollary in the first estimate of
$\|P_0\eu(T)\|$, \eqref{firstestimate}, to
 obtain the superconvergence estimate we sought. Note that, since $\Omega$ is convex, we can use the elliptic
  regularity inequality \eqref{ellipticregularity}.

\begin{theorem}
\label{thm:super} Assume that $u\in\mathcal{C}^1(0,T;H^{k+2}(\Omega))$ and $\vq\in\mathcal{C}^1(0,T;{\bf H}^{k+1}(\Omega))$. Assume also that
$\tau^*_K$ and $1/\tau^{\max}_K$ are bounded by $\mathsf {C}$. Then, for $k\ge1$,  we have that
\begin{alignat*}{1}
\|(u-u^*_h)(T)\|&\le\, C_3\,\sqrt{\log{\kappa}}\,h^{k+2}\,.
\end{alignat*}
 where the constant $C_3$,
  only depends on ${C}$, $\alpha$,  $\|u\|_{\mathcal{C}^1(H^{k+2})}$, and on $\|\vq\|_{\mathcal{C}^1(H^{k+1})}$.
\end{theorem}

Let us relate $\kappa$ to $T$ and the maximum diameter of the simplexes of the
mesh, $h$.  For  $\log\kappa>1$,
\[
  \kappa^{\alpha+1}<\,\kappa^{\alpha+1}\,\log\kappa= C_{k,d}^2\,T^{\alpha+1}/\rho^2\le C\,C_{k,d}^2\,T^{\alpha+1}/h^2,
\]
when the mesh is quasi-uniform. We then easily see that  $\log\kappa< C\,  \log(T
  h^{-2/(\alpha+1)})$  for  $\log\kappa>1$. Therefore,
$\sqrt{\log\kappa}\le   \max\{1, C\,  \sqrt{\log(T h^{-2/(\alpha+1)})}\}\,.$
\begin{proof}
From the first estimate of $\|P_0\eu(T)\|$, \eqref{firstestimate}, we have that
\begin{align*}
\|P_0\eu(T)\| \le& \; \mathsf{H}_1(\vTheta)\,(\|\eq\|_{L^\infty(L^2)}
    + \|e_\vq\|_{L^\infty(L^2)})
 + \mathsf{H}_2(\vTheta)\left(\|\partial_t \eu\|_{-\alpha}
    + \|e_{u_t}\|_{-\alpha}\right)
\\
\le& \; C\,h\,\frac{\sqrt{\log{\kappa}}}{\alpha+1}\,\Big(\|\eq\|_{L^\infty(L^2)}
    + \|e_\vq\|_{L^\infty(L^2)}
      +\|\partial_t \eu\|_{-\alpha}
    + \|e_{u_t}\|_{-\alpha}\Big)\,\|\theta\|,
\end{align*}
by the estimates of the dual solution of the previous lemma. Using these estimates in \eqref{ustar}, we obtain
\begin{alignat*}{1}
\|u-u^\star_h\|  \le&\; C\,h^{k+2}\,|u|_{H^{k+2}(\mathcal{T}_h)}+ \;  C\,h\,\frac{\sqrt{\log{\kappa}}}{\alpha+1}\,\Big(\|\eq\|_{L^\infty(L^2)}
    + \|e_\vq\|_{L^\infty(L^2)}\\
    & \qquad \qquad \qquad +\|\partial_t \eu\|_{-\alpha}
    + \|e_{u_t}\|_{-\alpha}\Big)
+C\,h\,\|\eq\|_{L^\infty(L^2)}.
\end{alignat*}
The result now follows by using the error estimates of Theorem \ref{thm:ee}. $\quad \qed$~~
\end{proof}


 \section{Summary and concluding remarks}\label{sec:extensions} {We have carried out the a priori error analysis of a semi-discrete
HDG method for the spatial discretization to problem \eqref{heatequation}. Assuming that the exact solution is sufficiently regular, we proved
optimal error estimates  of  the approximations to $u$ in the $L_\infty\bigr(0,T;L_2(\Omega)\bigr)$-norm and to $-\nabla u$
 in the $L_\infty\bigr(0,T;{\bf L}_2(\Omega)\bigr)$-norm over a regular triangular meshes.  Moreover, for  quasi-uniform meshes, by a simple elementwise postprocessing, we obtained a faster approximation for $u$ with a superconvergence rate.
 All the results obtained in this paper can be extended almost
word-by-word to other superconvergent HDG methods as well as to the mixed methods that fit the general formulation of the HDG methods; see
\cite{CockburnQiuShi}.

The devising of time-space fully discrete DG methods able to deal in an efficient manner with the memory term constitutes the subject of ongoing
research.}
\providecommand{\bysame}{\leavevmode\hbox to3em{\hrulefill}\thinspace} \providecommand{\MR}{\relax\ifhmode\unskip\space\fi MR }
\providecommand{\MRhref}[2]{%
  \href{http://www.ams.org/mathscinet-getitem?mr=#1}{#2}
} \providecommand{\href}[2]{#2}


\begin{thebibliography}{10}
\bibitem{Balakrishnan1985}
V. Balakrishnan, Anomalous diffusion in one dimension, {\em Phys. A}, {132}, (1985) 569-580.

\bibitem{BrezziFortin91}
F.~Brezzi and M.~Fortin, \emph{Mixed and hybrid finite element methods},
  {Springer Verlag}, 1991.

\bibitem{ChabaudCockburn12}  B.~Chabaud and B.~Cockburn,
Uniform-in-time superconvergence of {HDG} methods for the heat equation, {\em Math. Comp.}, 81, (2012) 107--129.



\bibitem{ChenLiuAnhTurner2011}
C-M. Chen, F. Liu, V. Anh and I. Turner,
Numerical methods for solving a two-dimensional variable-order anomalous sub-diffusion equation,
{\em Math.  Comp.}, 81, (2012) 345-366.

\bibitem{ChenLiuTurnerAnh2010}
  C-M. Chen, F. Liu, I. Turner and V. Anh, Numerical schemes and multivariate extrapolation of
            a two-dimensional anomalous sub-diffusion equation., {\em Numer. Algor.},
  {54}, (2010) 1--21.

\bibitem{CockburnGopalakrishnanLazarov09}
B.~Cockburn, J.~Gopalakrishnan and R.~Lazarov, {Unified hybridization of
  discontinuous {Galerkin}, mixed and continuous {Galerkin} methods for second
  order elliptic problems}, {\em SIAM J. Numer. Anal.}, {47}, (2009)
  1319--1365.

\bibitem{CockburnGopalakrishnanSayas09}
B.~Cockburn, J.~Gopalakrishnan and F.-J. Sayas, {A projection-based error
  analysis of {HDG} methods}, {\em Math. Comp.}, {79}, (2010) 1351--1367.

\bibitem{CockburnQiuShi}
B.~Cockburn, W.~Qiu and K.~Shi, {Conditions for superconvergence of {HDG}
  methods for second-order elliptic problems}, {\em Math. Comp.}, ~81, (2012) 1327-–1353.

\bibitem{CockburnQiuShiCurved}
\bysame, {Conditions for superconvergence of {HDG}
  methods on curvilinear elements for second-order elliptic problems}, {\em SIAM J.
  Numer. Anal.}, ~ 50, (2012) 1417-–1432.

\bibitem{CuestaLubichPalencia2006}
E. Cuesta, C. Lubich and C. Palencia, Convolution quadrature time discretization of fractional diffusive-wave equations, {\em Math. Comp.},~{75},
(2006) 673--696.

\bibitem{Cui2009} M. Cui, Compact finite difference method for
             the fractional diffusion equation, {\em J. Comput. Phys.}, 228,
 (2009)
7792--7804.

\bibitem{Cui2012}
\bysame, Convergence analysis of high-order compact alternating direction implicit schemes for the two-dimensional time fractional diffusion
equation, {\em Numer. Algor.}, 62, (2013) 383�-409.


\bibitem{GaoSun2011} G.G. Gao and Z.Z. Sun, A box-type scheme for fractional sub-diffusion
equation with Neumann boundary conditions, {\em J. Comput. Phys.}, 230, (2011) 6061-�6074.

\bibitem{GastaldiNochetto89}
L.~Gastaldi and R.H. Nochetto, {Sharp maximum norm error estimates for
  general mixed finite element approximations to second order elliptic
  equations}, {\em RAIRO Mod\'el. Math. Anal. Num\'er.}, {23}, (1989) 103--128.


   \bibitem{HenryWearne2000}
B. I. Henry and S. L. Wearne, Fractional reaction-diffusion, {\em Physica A}, {276}, (2000) 448--455.

 \bibitem{JinLazarovZhou2012} B. Jin, R. Lazarov and Z. Zhou, Error estimates for a semidiscrete finite element method for fractional
 order parabolic equations, {\em SIAM J. Numer. Anal.}, 51 (2013)  445�-466.


\bibitem{kilbas} A.A. Kilbas, H.M. Srivastava and  J.J. Trujillo, Theory and Applications of Fractional Differential
Equations, Volume 204 (North-Holland Mathematics Studies), 2006.


\bibitem{KirbySherwinCockburn}
R.M. Kirby, S.J. Sherwin and B.~Cockburn, {To {HDG} or to {CG}: A
  comparative study}, {\em J. Sci. Comput.}, 51, (2012) 183�-212.



\bibitem{LanglandsHenry2005}
T.~A.~M.~Langlands and B.~I.~Henry, The accuracy and stability of an implicit solution method for the fractional diffusion equation, {\em J.
Comput. Phys.}, {205}, (2005) 719--936.


\bibitem{LiuYangBurrage2009}
F.  Liu,  C. Yang,  and K. Burrage, Numerical method and analytical technique of the modified anomalous sub-diffusion equation with a nonlinear
source term, {\em Comput. Appl. Math.}, 231, (2009) 160-176.

\bibitem{LopezFernandezPalenciaSchaedle2006}
M.~L\'opez-Fern\'andez, C.~Palencia and A.~Sch\"adle, A spectral order method for inverting sectorial Laplace transforms, {\em SIAM J. Numer.
Anal.}, {44}, (2006) 1332--1350.

\bibitem{Mathai} A.M. Mathai, R. K, Saxena and  H. J. Haubold, The H-Function: Theory and Applications, Springer, 2011.


\bibitem{McLean2010} W. McLean, Regularity of solutions to a time-fractional diffusion
equation, \textit{ANZIAM J.}, {52},  (2010) 123--138.

\bibitem{McLean2012} \bysame, Fast summation by interval clustering for an
evolution equation with memory, {\em SIAM J. Sci. Comput.}, {34}, (2012)  A3039-A3056.


\bibitem{McLeanMustapha2007}
W. McLean,  and K. Mustapha, A second-order accurate numerical method for a fractional wave equation, {\em Numer. Math.}, {105}, (2007) 481--510.


\bibitem{McLeanMustapha2009}
\bysame, {Convergence analysis of a discontinuous Galerkin method for a sub-diffusion equation},
 {\em Numer. Algor.}, {52},
(2009) 69--88.


\bibitem{McLeanThomee2010L} W.
McLean and V. Thom\'ee, Numerical solution via Laplace transforms of a fractional order evolution equation, {\em J. Integral Equations Appl.}, 22,
(2010) 57--94.




\bibitem{MetzlerKlafter2000}
R. Metzler and J. Klafter, The random walk's guide to anomalous diffusion: a fractional dynamics approach, {\em Physics Reports} {339}, (2000)
1--77.

\bibitem{MetzlerKlafter2004}
\bysame, The restaurant at the end of the random walk: Recent developments in the description of anomalous transport by
fractional dynamics, {\em J. Phys. A}, {37}, (2004) R161--R208.


\bibitem{Mustapha2011}
K. Mustapha, An implicit finite difference time-stepping method for
            a sub-diffusion equation, with spatial discretization
            by finite elements, {\em IMA J. Numer. Anal.}, {31}, (2011)
719--739.

\bibitem{MustaphaAlMutawa}
K. Mustapha and J. AlMutawa, A finite difference  method for an anomalous sub-diffusion equation, theory and applications, {\em Numer. Algor.},
 61, (2012) 525--543

\bibitem{MustaphaMcLean2011}
K. Mustapha and  W. McLean, Piecewise-linear, discontinuous Galerkin method for a fractional diffusion equation, {\em Numer. Algor.}, {56},
(2011) 159--184.

\bibitem{MustaphaMcLean20xx}
\bysame,  Uniform convergence for a discontinuous Galerkin, time stepping method applied to a fractional diffusion equation, {\em IMA J. Numer.
Anal.}, {32},  (2012) 906--925.

\bibitem{MustaphaMcLean2013}
\bysame, Superconvergence of a discontinuous Galerkin method for fractional diffusion and wave equations, {\em  SIAM J. Numer. Anal.},  51,
(2013) 491--515.


\bibitem{MustaphaSchoetzau2013} K. Mustapha and  D. Sch\"otzau, Well-posedness of $hp-$version discontinuous Galerkin methods for fractional diffusion wave equations, {\em  IMA J. Numer. Anal.}, (2013), accepted,  doi: 10.1093/imanum/drt048.

 \bibitem{NguyenPeraireCockburnIcosahom09}
N.C. Nguyen, J.~Peraire and B.~Cockburn, {Hybridizable discontinuous
  {Galerkin} methods}, {\em Proceedings of the International Conference on Spectral
  and High Order Methods} (Trondheim, Norway), Lect. Notes Comput. Sci. Engrg.,
  Springer Verlag, June 2009.

\bibitem{podlubny} I. Podlubny, Fractional Differential Equations,
Academic Press, San Diego, 1999.


\bibitem{Quintana-MurilloYuste2011} J. Quintana-Murillo and S.B. Yuste, An explicit difference method for solving
fractional diffusion and diffusion-wave equations in the Caputo form, {\em  J. Comput. Nonlin. Dyn.}, 6, (2011) 021014.


\bibitem{SchaedleLopezFernandezLubich2006}
A. Sch\"adle, M. L\'opez-Fernandez and C. Lubich, Fast and oblivious convolution quadrature, {\em SIAM J. Sci. Comput.}, {28}, (2006) 421--438.

\bibitem{SchneiderWyss1989}
W.~R.~Schneider,  W.~Wyss, Fractional diffusion and wave equations, {\em J. Math. Phys.}, {30},  (1989) 134--144.


\bibitem{SmithMorrison} P. Smith, I. Morrison, K. Wilson, N. Fernandez and R. Cherry, Anomalous diffusion of major histocompatability complex class I molecules on HeLa
cells determined by single particle tracking, {\em Biophys. J.}, {76}, (1999) 3331--3344.





\bibitem{Stenberg88}
R.~Stenberg, {A family of mixed finite elements for the elasticity
  problem}, {\em Numer. Math.}, {53}, (1988) 513--538.

\bibitem{Stenberg91}
\bysame, {Postprocessing schemes for some mixed finite elements},
  {\em RAIRO Mod\'el. Math. Anal. Num\'er.}, {25}, (1991) 151--167.

\bibitem{Tarasov} V. E. Tarasov,  Fractional Dynamics: Applications of Fractional Calculus to Dynamics of Particles,
Fields and Media (Nonlinear Physical Science), Springer, 2011



\bibitem{Wyss1986}
W. Wyss, Fractional diffusion equation, {\em J. Math. Phys.}, {27}, (1986)
 2782-2785.


\bibitem{Yuste2006}
S. B. Yuste,  Weighted average finite difference methods
            for fractional diffusion equations,
{\em  J. Comput. Phys.}, {216}, (2006) 264--274.

\bibitem{YusteAcedo2005}
S. B. Yuste and L.  Acedo,  An explicit finite difference method and a new von Neumann-type stability analysis for fractional diffusion equations,
{\em  SIAM J. Numer. Anal.}, {42}, (2005) 1862--1874.

\bibitem{YusteQuintana2009}
S.B. Yuste and J. Quintana-Murillo, On Three Explicit Difference Schemes for Fractional Diffusion and Diffusion-Wave Equations, {\em Phys. Scripta
T136}, (2009) 014025.

\bibitem{ZhangSun2011} Y.-N. Zhang and
Z.-Z. Sun, Alternating direction implicit schemes for the two-dimensional fractional sub-diffusion equation,
 {\em  J. Comput. Phys.}, {230}, (2011) 8713--8728.

\bibitem{ZhuangLiuAnhTurner2008}
P. Zhuang, F. Liu, V. Anh and  I. Turner, New solution and analytical techniques of the implicit numerical methods for the anomalous sub-diffusion
equation, {\em SIAM J.  Numer. Anal.}, 46, (2008)  1079-1095.




\bibitem{ZhuangLiuAnhTurner2009}
\bysame, Stability and convergence of an implicit numerical method for the nonlinear fractional
reaction-sub-diffusion process, {\em IMA J.  Appl. Math.}, 74, (2009) 645-667.



\end{thebibliography}
\end{document}